\documentclass[12pt,a4paper]{article}%
\usepackage[utf8]{inputenc}
\usepackage{amsmath}
\usepackage{amsfonts}
\usepackage{amssymb}
\usepackage{graphicx}
\usepackage[space]{grffile}
\usepackage{hyperref}
\usepackage{layout}
\usepackage{longtable}%
\providecommand{\U}[1]{\protect\rule{.1in}{.1in}}
\newtheorem{theorem}{Theorem}
\newtheorem{theorem*}{Example}

\newtheorem{conjecture}[theorem]{Conjecture}

\newtheorem{definition}[theorem]{Definition}

\newtheorem{lemma}[theorem]{Lemma}

\newtheorem{remark}[theorem]{Observation}

\newenvironment{proof}[1][Proof]{\noindent\textbf{#1.} }{\ \hfill \rule{0.5em}{0.5em}\bigskip}
\graphicspath{}
\begin{document}

\title{Normal $5$-edge coloring of some more snarks superpositioned by the Petersen graph}
\author{Jelena Sedlar$^{1,3}$,\\Riste \v Skrekovski$^{2,3}$ \\[0.3cm] {\small $^{1}$ \textit{University of Split, Faculty of civil
engineering, architecture and geodesy, Croatia}}\\[0.1cm] {\small $^{2}$ \textit{University of Ljubljana, FMF, 1000 Ljubljana,
Slovenia }}\\[0.1cm] {\small $^{3}$ \textit{Faculty of Information Studies, 8000 Novo
Mesto, Slovenia }}\\[0.1cm] }
\maketitle

\begin{abstract}
\ A normal $5$-edge-coloring of a cubic graph is a coloring such that for
every edge the number of distinct colors incident to its end-vertices is $3$
or $5$ (and not $4$). The well known Petersen Coloring Conjecture is
equivalent to the statement that every bridgeless cubic graph has a normal
$5$-edge-coloring. All $3$-edge-colorings of a cubic graph are obviously
normal, so in order to establish the conjecture it is sufficient to consider
only snarks. In our previous paper [\textit{J. Sedlar, R. \v{S}krekovski,
Normal 5-edge-coloring of some snarks superpositioned by the Petersen graph,
Applied Mathematics and Computation 467 (2024) 128493}], we considered
superpositions of any snark $G$ along a cycle $C$ by two simple supervertices
and by the superedge obtained from the Petersen graph, but only for some of
the possible ways of connecting supervertices and superedges. The present
paper is a continuation of that paper, herein we consider superpositions by
the Petersen graph for all the remaining connections and establish that for
all of them the Petersen Coloring Conjecture holds.

\end{abstract}



\section{Introduction}

A (\emph{proper})\emph{ }$k$\emph{-edge-coloring} of a graph $G$ is any
mapping $\sigma:E(G)\rightarrow\{1,\ldots,k\}$ such that any pair of adjacent
edges of $G$ receives distinct colors by $\sigma$. Let $\sigma$ be a
$k$-edge-coloring of $G$ and $v\in V(G),$ then by $\sigma(v)$ we denote the
set of colors of all edges incident to $v.$

\begin{definition}
Let $G$ be a bridgeless cubic graph, $\sigma$ a proper edge-coloring of $G$
and $uv$ an edge of $G.$ The edge $uv$ is \emph{poor} \emph{(}resp\emph{.}
\emph{rich)} in $\sigma$ if $\left\vert \sigma(u)\cup\sigma(v)\right\vert =3$
\emph{(}resp\emph{.} $\left\vert \sigma(u)\cup\sigma(v)\right\vert
=5$\emph{).}
\end{definition}

A proper edge coloring of a cubic graph $G$ is \emph{normal}, if every edge of
$G$ is poor or rich. Normal edge colorings were introduced by Jaeger in
\cite{Jaeger1985}. The \emph{normal chromatic index} of a cubic graph $G,$
denoted by $\chi_{N}^{\prime}(G),$ is the smallest number $k$ such that $G$
has a normal $k$-edge coloring. Notice that $\chi_{N}^{\prime}(G)$ is at least
$3$ and it never equals $4$. It is known that the well known Petersen Coloring
Conjecture can be restated as follows \cite{Jaeger1985}.

\begin{conjecture}
\label{Con_normal}Let $G$ be a bridgeless cubic graph, then $\chi_{N}^{\prime
}(G)\leq5.$
\end{conjecture}

\noindent Since any proper $3$-edge-coloring of a cubic graph $G$ is a normal
edge coloring in which every edge is poor, Conjecture \ref{Con_normal}
obviously holds for every $3$-edge-colorable graph. According to Vizing's
theorem, every cubic graph is either $3$-edge-colorable or $4$-edge-colorable,
so to prove Conjecture \ref{Con_normal} it remains to establish that it holds
for all bridgeless cubic graphs which are not $3$-edge-colorable.

\paragraph{Superpositioning snarks.}

Cubic graphs which are not $3$-edge-colorable are usually considered under the
name of snarks \cite{MazzuoccoloStephen,NedelaSkovieraSurvey}. In order to
avoid trivial cases, the definition of snark usually includes some additional
requirements on connectivity. Since in this paper such requirements are not
essential, we will go with the broad definition of a \emph{snark} being any
bridgeless cubic graph which is not $3$-edge-colorable. The existence of a
normal $5$-edge-coloring for some families of snarks has already been
established, see for example \cite{MazzuccoloLupekhine, Hagglund2014}.

The superposition is the most general known method for constructing new snarks
from the existing ones \cite{Adelson1973, Descartes1946, KocholSuperposition,
Kochol2, MacajovaRevisited}. In this paper we consider some superpositioned
snarks, so let us first introduce this method.

\begin{definition}
A \emph{multipole} $M=(V,E,S)$ consists of a set of vertices $V=V(M)$, a set
of edges $E=E(M)$, and a set of semiedges $S=S(M)$. A semiedge is incident
either to one vertex or to another semiedge in which case a pair of incident
semiedges forms a so called \emph{isolated edge} within the multipole.
\end{definition}

\begin{figure}[h]
\begin{center}%
\begin{tabular}
[t]{llll}%
$A:$ & \raisebox{-0.9\height}{\includegraphics[scale=0.6]{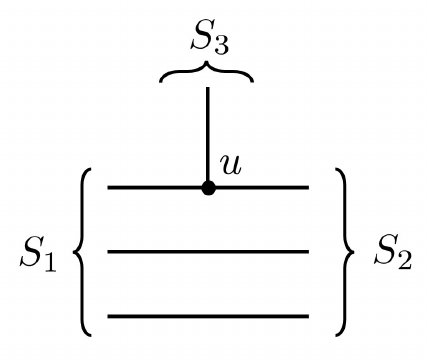}}\hspace
{2cm} & $A^{\prime}:$ &
\raisebox{-0.9\height}{\includegraphics[scale=0.6]{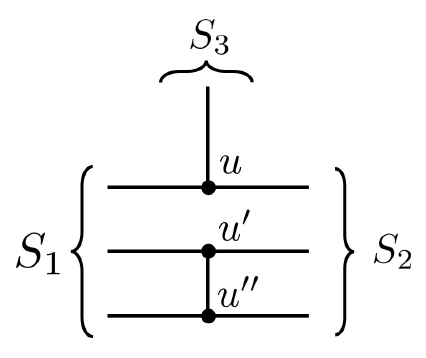}}
\end{tabular}
\end{center}
\caption{Supervertices $A$ and $A^{\prime},$ with connectors $S_{1},$ $S_{2}$
and $S_{3}$, where $S_{1}$ is $1$-connector and $S_{2}$ and $S_{3}$ are
$3$-connectors.}%
\label{Fig_supervertex}%
\end{figure}

For an illustration of multipole see Figure \ref{Fig_supervertex}. If
$\left\vert S(M)\right\vert =k,$ then $M$ is also called a $k$\emph{-pole}.
Throughout the paper we consider only \emph{cubic} multipoles, i.e. multipoles
in which every vertex is incident to precisely three edges or semiedges. If
the set of semiedges $S$ is partitioned into $n$ subsets $S_{i}$ such that
$\left\vert S_{i}\right\vert =k_{i}$, this results with a $(k_{1},\ldots
,k_{n})$\emph{-pole} $M=(V,E,S_{1},\ldots,S_{n}).$ Sets $S_{i}$ are called
\emph{connectors} of $M.$ A connector $S_{i}$ which has $k_{i}$ semiedges is
also called $k_{i}$\emph{-connector}.

\begin{definition}
A \emph{supervertex} \emph{(}resp\emph{.} \emph{superedge)} is a multipole
with three \emph{(}resp\emph{.} two\emph{)} connectors.
\end{definition}

Let $G$ be a snark and $u,v$ a pair of non-adjacent vertices of $G.$ A
superedge $G_{u,v}$ is obtained from $G$ by removing vertices $u$ and $v,$ and
then replacing three edges incident to a vertex $u$ (resp. $v$) in $G$ by
three semiedges in $G_{u,v}$ which form a connector.

\begin{definition}
A \emph{proper }superedge is any superedge $G_{u,v},$ where $G$ is a snark, or
an isolated edge.
\end{definition}

\noindent The definition of proper superedge from \cite{KocholSuperposition}
is much wider, but this simple definition suffices for the purposes of the
present paper. Now, for a cubic graph $G=(V,E)$ we define a function
$\mathcal{V}$ which maps a vertex $v\in V$ to a supervertex
$\mathcal{\mathcal{V}}(v),$ and a function $\mathcal{E}$ which maps an edge
$e\in E$ to a superedge $\mathcal{E}(e)$. A superposition $G(\mathcal{V}%
,\mathcal{E})$ is obtained if the following holds: semiedges of a connector in
$\mathcal{V}(v)$ are identified with semiedges of a connector in
$\mathcal{E}(e)$ if and only if $e$ is incident to $v$ in $G.$ Of course, for
this to be possible the corresponding connectors of $\mathcal{V}(v)$ and
$\mathcal{E}(e)$ must be of equal size. Notice that a superposition
$G(\mathcal{V},\mathcal{E})$ is again a cubic graph, and it is called
\emph{proper} if every superedge $\mathcal{E}(e)$ is proper. Moreover, we may
assume that some of the vertices and edges of $G$ are superpositioned by
themselves, formally such a vertex (resp. an edge) is superpositioned by a
trivial supervertex (resp. a superedge) which consists of a single vertex with
three incident semiedges (resp. an isolated edge). The following theorem is
established in \cite{KocholSuperposition}, mind that it is stated there for
snarks of girth $\geq5$ which are cyclically $4$-edge connected, but it also
holds for snarks of smaller girth.

\begin{theorem}
\label{Tm_Kochol}For a snark $G,$ a proper superposition $G(\mathcal{V}%
,\mathcal{E})$ is also a snark.
\end{theorem}

\paragraph{Normal colorings of superpositioned snarks.}

Normal $5$-edge-colorings of a family of superpositioned snarks are considered
in \cite{SedSkrePaper1} and \cite{SedSkePaper2}. In \cite{SedSkrePaper1}, a
snark $G$ is superpositioned along a cycle $C$ by supervertices $A$ or
$A^{\prime}$ and superedges $(P_{10})_{u,v},$ where $P_{10}$ denotes the
Petersen graph. Notice that the three semiedges of the corresponding
$3$-connectors of the supervertex $A$ or $A^{\prime}$ and the superedge
$(P_{10})_{u,v}$ can be identified in $3!=6$ ways, and in \cite{SedSkrePaper1}
only one particular way of semiedge identification is considered. Assuming
that $G$ has a normal $5$-edge-coloring $\sigma$, it is established there that
a superposition also has a normal $5$-edge-coloring, but only for even length
cycles $C.$ In the case of odd length cycles $C,$ it is established that in
some cases a normal $5$-edge-coloring of a superposition cannot be constructed
without changing the colors on the edges of $G$ outside of the cycle $C.$

In \cite{SedSkePaper2}, similar superpositions are considered, where the
Flower snarks $J_{r}$ are taken as superedges. To construct a normal
$5$-edge-coloring of such superpositioned snarks, a sufficient condition is
developed for the existence of a desired coloring for a superposition by any
superedge $H_{u,v}$, and the application of this condition yields a normal
$5$-edge-coloring of all superpositions by Flower snark superedges
$(J_{r})_{u,v}$.

Let us remark that the condition from \cite{SedSkePaper2} does not apply to
superpositions by the Petersen superedge $(P_{10})_{u,v},$ since it requires
that the vertices $u$ and $v$ are on the distance $\geq3$ which is not the
case with $P_{10}.$ Thus, the present paper is a continuation of
\cite{SedSkrePaper1}, herein we construct a normal $5$-edge-colorings for all
remaining superpositions by the Petersen superedge, i.e. for all remaining
ways of semiedge identification not considered in \cite{SedSkrePaper1}.

\section{Preliminaries}

Let us first introduce some necessary notions regarding multipoles and their
colorings. Let $M=(V,E,S)$ be a multipole and $V^{\prime}\subseteq V.$ A
multipole $M^{\prime}=(V^{\prime},E^{\prime},S^{\prime})$ is an \emph{induced
submultipole} of $M$ if $E^{\prime}$ consists of all edges of $M$ with both
end-vertices in $V^{\prime},$ and $S^{\prime}$ consists of all semiedges of
$M$ with the only end-vertex in $V^{\prime}$ and of a semiedge for each edge
of $M$ with precisely one end in $V^{\prime}.$ An induced submultipole is also
denoted by $M^{\prime}=M[V^{\prime}].$ For a normal $5$-edge-coloring $\sigma$
of $M,$ the \emph{restriction} $\sigma^{\prime}=\left.  \sigma\right\vert
_{M^{\prime}}$ of $\sigma$ to $M^{\prime}$ is defined by $\sigma^{\prime
}(e)=\sigma(e)$ for $e\in E^{\prime},$ $\sigma^{\prime}(s)=\sigma(s)$ for
$s\in S^{\prime}\cap S$ and $\sigma^{\prime}(s)=\sigma(e_{s})$ for $s\in
S^{\prime}\backslash S$ where $e_{s}$ denotes the edge such that $s$ is a
semiedge of $e_{s}.$ Let $M$ be a multipole, $M_{i}$ an induced submultipole
of $M$ with a normal $5$-edge-coloring $\sigma_{i}$ for $i=1,\ldots,k.$
Colorings $\sigma_{i}$ are \emph{compatible}, if there exists a normal
$5$-edge-coloring $\sigma^{\prime}$ of an induced submultipole $M^{\prime
}=M[\cup_{i=1}^{k}V_{i}]$ such that $\left.  \sigma^{\prime}\right\vert
_{M^{\prime}}=\sigma_{i}$ for every $i=1,\ldots,k.$

\begin{figure}[h]
\begin{center}%
\begin{tabular}
[t]{llll}%
a) & \raisebox{-0.9\height}{\includegraphics[scale=0.6]{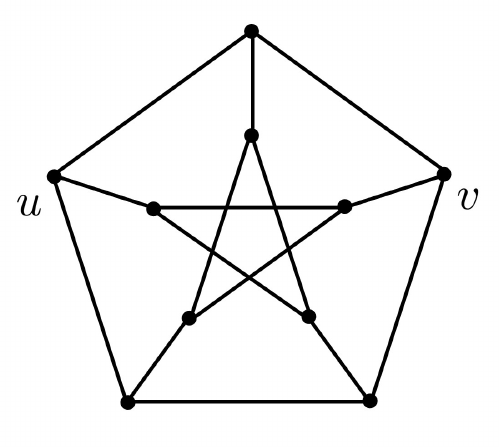}} & b) &
\raisebox{-0.9\height}{\includegraphics[scale=0.6]{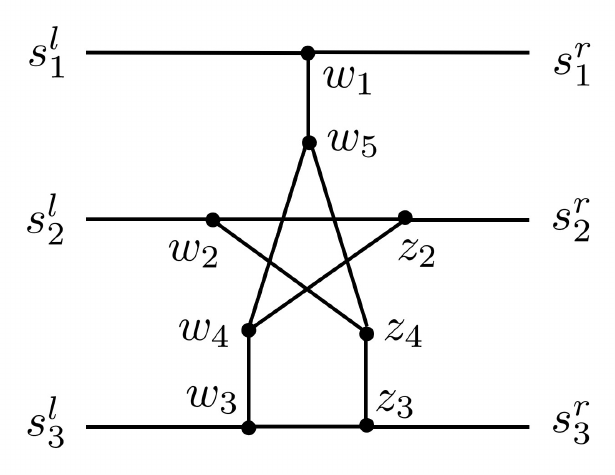}}
\end{tabular}
\end{center}
\caption{The figure shows: a) the Petersen graph $P_{10}$ with a pair of
nonadjacent vertices $u$ and $v$, b) the corresponding superedge
$(P_{10})_{u,v}.$}%
\label{Fig_PetersenSuperedge}%
\end{figure}

A \emph{normal }$5$\emph{-edge-coloring} of a cubic multipole is any coloring
of its edges and semiedges by $5$ colors such that every edge of $M$ is
normal. Let $M$ be a cubic multipole, $s$ a semiedge of $M$ and $\sigma$ a
normal $5$-edge-coloring of $M,$ the \emph{color scheme} $\sigma\lbrack s]$ of
a semiedge $s$ is defined by $\sigma\lbrack s]=(i,\{j,k\})$ where
$\sigma(s)=i$ and $\{j,k\}$ is the set of the remaining two colors incident to
the end-vertex of $s.$ A pair of color schemes $(i_{1},\{j_{1},k_{1}\})$ and
$(i_{2},\{j_{2},k_{2}\})$ is \emph{consistent} if $i_{1}=i_{2}$ and
$\{j_{1},k_{1}\}$ is equal either to $\{j_{2},k_{2}\}\ $or\ to
$\{1,2,3,4,5\}\backslash\{i_{2},j_{2},k_{2}\}.$ For consistent color schemes
we also write $(i_{1},\{j_{1},k_{1}\})\approx(i_{2},\{j_{2},k_{2}\}).$

Let us now focus our attention to the type of superpositions considered in
this paper. We will consider any snark $G$ and a cycle $C$ in it. Vertices of
$C$ are superpositioned by supervertices $A$ and $A^{\prime}$ from Figure
\ref{Fig_supervertex}, and edges of $C$ are superpositioned by the superedge
obtained from the Petersen graph. Namely, let $P_{10}$ denote the Petersen
graph shown in Figure \ref{Fig_PetersenSuperedge}.a), $u,v$ a pair of
nonadjacent vertices of $P_{10},$ and $(P_{10})_{u,v}$ a corresponding proper
superedge which is illustrated by Figure \ref{Fig_PetersenSuperedge}.b). A
connector arising from removing $u$ (resp. $v$) is called the \emph{left}
(resp. \emph{right}) connector of $(P_{10})_{u,v}$ and its semiedges are
called the \emph{left} (resp. \emph{right}) semiedges and denoted by
$s_{j}^{l}$ (resp. $s_{j}^{r}$) for $j=1,2,3,$ as it is shown in Figure
\ref{Fig_PetersenSuperedge}.b).

\begin{figure}[h]
\begin{center}
\includegraphics[scale=0.6]{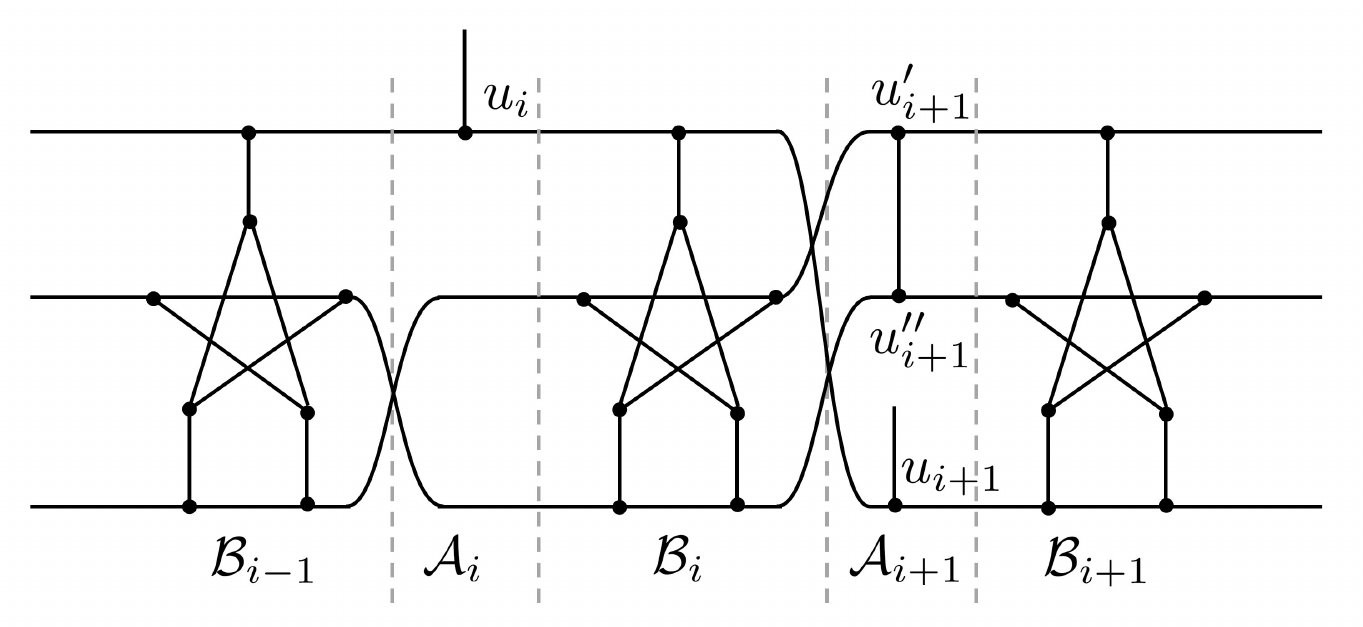}
\end{center}
\caption{A superposition of the vertices $u_{i}$ and $u_{i+1}$ of a cycle $C$
in a snark $G$ by superedges $\mathcal{A}_{i}=A$ and $\mathcal{A}%
_{i+1}=A^{\prime},$ and edges $e_{i-1}=u_{i-1}u_{i},$ $e_{i}=u_{i}u_{i+1}$ and
$e_{i+1}=u_{i+1}u_{i+2}$ by superedges $\mathcal{B}_{i-1},$ $\mathcal{B}_{i}$
and $\mathcal{B}_{i+1}$ so that $p_{i-1}=(1,3,2),$ $p_{i}=(3,1,2),$ $d_{i}=1,$
and $d_{i+1}=3.$}%
\label{Fig_superpositionPermutation}%
\end{figure}

Now, assume that a cycle $C$ in a snark $G$ is denoted by $C=u_{0}u_{1}\cdots
u_{g-1}u_{0}.$ Denote the edges of the cycle $C$ by $e_{i}=u_{i}u_{i+1}$ for
$i=0,\ldots,g-1,$ where indices are taken modulo $g$. Also, let $v_{i}$ denote
the neighbor of $u_{i}$ distinct from $u_{i-1}$ and $u_{i+1},$ and let
$f_{i}=u_{i}v_{i}.$ Let $G_{C}(\mathcal{A},\mathcal{B})$ be a superposition
such that for vertices and edges of a cycle $C$ in $G$ it holds that
$\mathcal{A}(u_{i})\in\{A,A^{\prime}\}$, $\mathcal{\mathcal{B}}(e_{i}%
)=(P_{10})_{u,v},$ and all other vertices and edges of $G$ are superpositioned
by themselves.\ For short, we will denote $\mathcal{A(}u_{i})=\mathcal{A}_{i}$
and $\mathcal{B}(e_{i})=\mathcal{B}_{i}.$ Also, if $\mathcal{A}_{i}=A$ then
its only vertex is denoted by $u_{i}$, and if $\mathcal{A}_{i}=A^{\prime}$
then its three vertices are denoted by $u_{i},$ $u_{i}^{\prime}$ and
$u_{i}^{\prime\prime}$.

In order to describe more precisely how semiedges of $\mathcal{\mathcal{B}%
}_{i-1},$ $\mathcal{A}_{i}$ and $\mathcal{B}_{i}$ are identified, which is
illustrated by Figure \ref{Fig_superpositionPermutation}, each superedge
$\mathcal{B}_{i}$ is associated with a permutation $p_{i}$ of the set
$\{1,2,3\}$ which is called the \emph{right semiedge permutation} and a
$d_{i}\in\{1,2,3\}$ which is called the \emph{dock}. Then, the connection of
$\mathcal{B}_{i-1}$ and $\mathcal{B}_{i}$ through $\mathcal{A}_{i}$ arises so
that a right semiedge $s_{j}^{r}$ of $\mathcal{B}_{i-1}$ is identified with a
left semiedge $s_{p_{i-1}(j)}^{l}$ of $\mathcal{B}_{i},$ and then the edge
arising from identifying $s_{j}^{r}$ and $s_{p_{i-1}(j)}^{l}$ for
$p_{i-1}(j)=d_{i}$ is subdivided by the vertex $u_{i}$ of $\mathcal{A}_{i}.$
If $\mathcal{A}_{i}=A^{\prime},$ the remaining two edges arising from the
identification are subdivided by $u_{i}^{\prime}$ and $u_{i}^{\prime\prime},$
and an edge $u_{i}^{\prime}u_{i}^{\prime\prime}$ inserted afterwards.

In \cite{SedSkrePaper1}, this class of superpositions of $G$ by the Petersen
graph is also considered, but only superpositions such that $p_{i}(1)=1$ and
$d_{i}=1$ for every $i=0,\ldots,g-1.$ It is established there that for such
superpositions a normal $5$-edge-coloring can be constructed from a normal
$5$-edge-coloring of $G$ by preserving colors outside of the cycle $C,$ but
only for even length cycles $C$. For odd length cycles $C,$ it is established
that at least in one case a normal $5$-edge-coloring of a superposition cannot
be constructed without changing colors of edges outside the cycle $C.$ In this
paper, we will consider all the remaining superpositions of the mentioned
class, i.e. superpositions with either $p_{i}(1)\not =1$ or $d_{i}\not =1$ for
some $i\in\{0,\ldots,g-1\}.$

\begin{figure}[h]
\begin{center}
\includegraphics[scale=0.7]{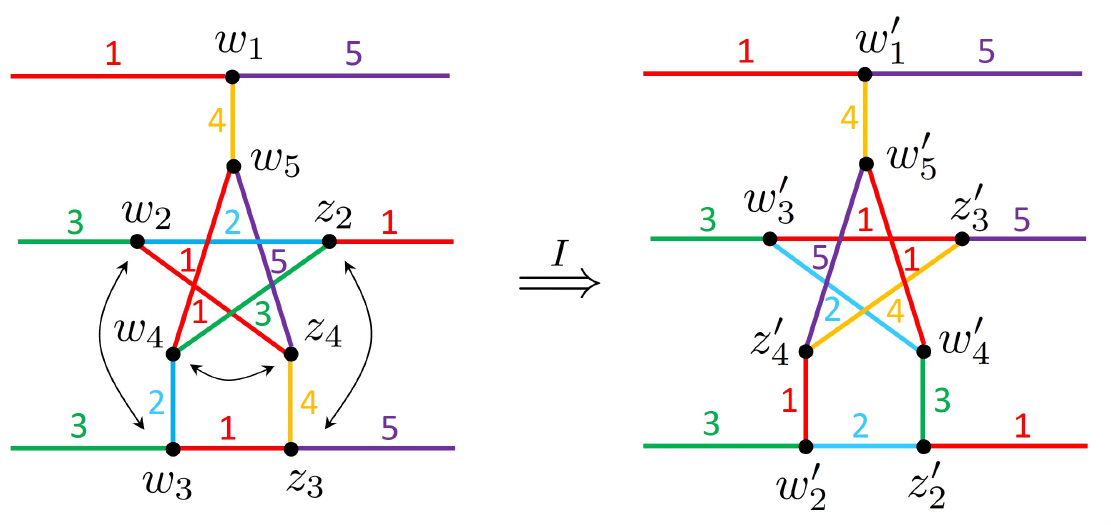}
\end{center}
\caption{The figure shows a normal $5$-edge-coloring $\sigma$ of
$(P_{10})_{u,v}$ and the application of $I$ to $(P_{10})_{u,v}$ which yields a
normal $5$-edge-coloring $I(\sigma).$}%
\label{Fig_isomorphism}%
\end{figure}

We also need the following isomorphism of the superedge $(P_{10})_{u,v}.$
Assuming vertex notation as in Figure \ref{Fig_PetersenSuperedge}, we define
an isomorphism $I$ of $(P_{10})_{u,v}$ by the permutation of vertices $(w_{2}$
$w_{3})(z_{2}$ $z_{3})(w_{4}$ $z_{4})$ considered as a product of three
transpositions, see Figure \ref{Fig_isomorphism}. Notice that $I$ is
involution, i.e. $I=I^{-1}.$ For the simplicity sake, we will denote $I(v)$ by
$v^{\prime}.$ If $\sigma$ is a normal edge-coloring of $(P_{10})_{u,v}$ and we
assume that in the isomorphism $I$ the colors of edges are mapped together
with edges, this yields another normal edge-coloring of $(P_{10})_{u,v}.$
Formally, for a given normal $5$-edge-coloring $\sigma$ of $(P_{10})_{u,v}$ we
define an edge-coloring $\sigma_{I}$ of $(P_{10})_{u,v}$ by $\sigma
_{I}(x)=\sigma(I(x))$ for any (semi)edge $x$ of $(P_{10})_{u,v}.$ The
following observation on $\sigma_{I}$ is illustrated by Figure
\ref{Fig_isomorphism}.

\begin{remark}
If $\sigma$ is a normal $5$-edge-coloring of $(P_{10})_{u,v},$ then
$\sigma_{I}$ is also a normal $5$-edge-coloring of $(P_{10})_{u,v}$.
\end{remark}

Finally, a path connecting a pair of semiedges of a multipole is called a
\emph{Kempe }$(i,j)$\emph{-chain} in an edge-coloring $\sigma$ if the colors
$i$ and $j$ alternate along the edges of the path. For example, both colorings
of $(P_{10})_{u,v}$ from Figure \ref{Fig_isomorphism} have a Kempe
$(3,2)$-chain connecting left semiedges $s_{2}^{l}$ and $s_{3}^{l}.$

\section{Auxiliary results}

\begin{figure}[h]
\begin{center}%
\begin{tabular}
[t]{llll}%
a) & \raisebox{-0.9\height}{\includegraphics[scale=0.6]{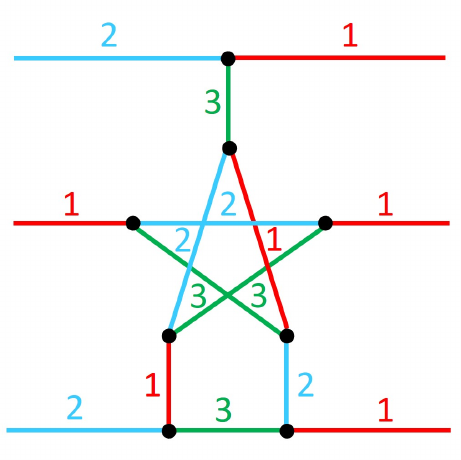}} & b) &
\raisebox{-0.9\height}{\includegraphics[scale=0.6]{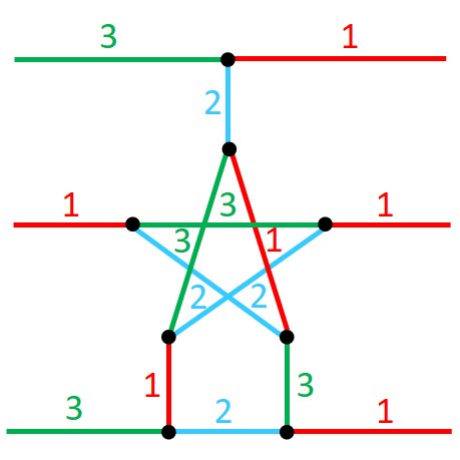}}
\end{tabular}
\end{center}
\caption{The coloring: a) $R(1,2,3)$ and b) $\bar{R}(1,2,3)$ of the superedge
$(P_{10})_{u,v}.$}%
\label{Fig_rightColoring}%
\end{figure}For the superedge $(P_{10})_{u,v}$ we define a normal
$5$-edge-coloring $R(1,2,3)$ as in Figure \ref{Fig_rightColoring}. Notice that
the coloring $R(1,2,3)$ has a Kempe $(2,3)$-chain $P^{l}$ connecting left
semiedges $s_{1}^{l}$ and $s_{3}^{l}.$ A \emph{complementary} coloring
$\bar{R}(1,2,3)$ of $R(1,2,3)$ is obtained from $R(1,2,3)$ by swapping colors
$2$ and $3$ along the Kempe chain $P^{l}.$ Notice that the complementary
coloring $\bar{R}(1,2,3)$ remains a normal $5$-edge-coloring of $(P_{10}%
)_{u,v}.$ Finally, let $c$ be a permutation of the set $\{1,2,3,4,5\}$ of five
colors such that $c(i)=c_{i}.$ Then the coloring $R(c_{1},c_{2},c_{3})$ and
its complementary coloring $\bar{R}(c_{1},c_{2},c_{3})$ are obtained by
applying color permutation $c$ to $R(1,2,3)$ and $\bar{R}(1,2,3)\mathrm{,}$
respectively. Notice that $\bar{R}(1,2,3)=R(1,3,2)\mathrm{,}$ and consequently
$\bar{R}(c_{1},c_{2},c_{3})=R(c_{1},c_{3},c_{2})$ for any permutation $c.$ By
$R_{I}(c_{1},c_{2},c_{3})$ and $\bar{R}_{I}(c_{1},c_{2},c_{3})$ we denote the
coloring obtained by applying the isomorphism $I$ of $(P_{10})_{u,v}$ to
$R(c_{1},c_{2},c_{3})$ and $\bar{R}(c_{1},c_{2},c_{3})\mathrm{,}$ respectively.

\begin{remark}
\label{Obs_rightPoor}Coloring $R(c_{1},c_{2},c_{3})$ of $(P_{10})_{u,v}$ and
its variants $\bar{R}(c_{1},c_{2},c_{3}),$ $R_{I}(c_{1},c_{2},c_{3}),$
$\bar{R}_{I}(c_{1},c_{2},c_{3})$ use only three colors, so they contain $9$
poor edges.
\end{remark}

\begin{figure}[h]
\begin{center}
\includegraphics[scale=0.6]{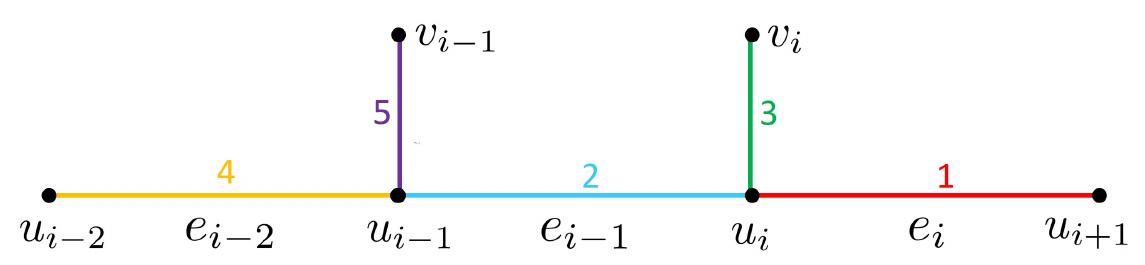}
\end{center}
\caption{A normal $5$-edge coloring $\sigma$ of the edges of $G$ incident to
vertices $u_{i-1}$ and $u_{i}$.}%
\label{Fig_coloringCycle}%
\end{figure}

Notice that the coloring $R$ of the superedge $(P_{10})_{u,v}$, and its
variants $\bar{R}\mathrm{,}$ $R_{I}\mathrm{,}$ $\bar{R}_{I}$, are defined for
any three colors $c_{1},$ $c_{2},$ $c_{3}$. Which three colors will be used in
a superedge $\mathcal{B}_{i}=(P_{10})_{u,v}$ within a superposition $\tilde
{G}=G_{C}(\mathcal{A},\mathcal{B})$ depends on the normal $5$-edge-coloring
$\sigma$ of $G.$ Namely, let $G$ be a snark with a normal $5$-edge-coloring
$\sigma,$ $C$ a cycle in $G$ and $\tilde{G}=G_{C}(\mathcal{A},\mathcal{B})$ a
superposition of $G.$ A coloring $R(c_{1},c_{2},c_{3})$ (resp. $\bar{R}%
(c_{1},c_{2},c_{3})\mathrm{,}$ $R_{I}(c_{1},c_{2},c_{3})\mathrm{,}$ $\bar
{R}_{I}(c_{1},c_{2},c_{3})$) of $\mathcal{B}_{i}$ is called $\sigma
$\emph{-colored} if $c_{1}=\sigma(e_{i}),$ $c_{2}=\sigma(e_{i-1})$ and
$c_{3}=\sigma(f_{i}),$ and it is denoted simply by $R$ (resp. $\bar
{R}\mathrm{,}$ $R_{I}\mathrm{,}$ $\bar{R}_{I}$). For example, if a normal
$5$-edge-coloring $\sigma$ of $G$ is as in Figure \ref{Fig_coloringCycle},
then the coloring $R(1,2,3)$ from Figure \ref{Fig_rightColoring} is a $\sigma
$-colored coloring of $\mathcal{B}_{i}.$

A $\sigma$-colored coloring $R$ and/or its variants $\bar{R}\mathrm{,}$
$R_{I}\mathrm{,}$ $\bar{R}_{I}$ have important properties for the construction
of normal $5$-edge-colorings of a superpositioned snark, so let us introduce
those properties. A normal $5$-edge-coloring $\tilde{\sigma}_{i}$ of
$\mathcal{B}_{i}$ is \emph{right-side }$\sigma$\emph{-monochromatic} if
$\tilde{\sigma}_{i}[s_{j}^{r}]\approx(\sigma(e_{i}),\{\sigma(e_{i-1}%
),\sigma(f_{i})\})$ for every $j\in\{1,2,3\}.$ For example, if a coloring
$\sigma$ of edges in $G$ incident to vertices $u_{i}$ is as in Figure
\ref{Fig_coloringCycle}, then the coloring $R(1,2,3)$ from Figure
\ref{Fig_rightColoring} is right-side $\sigma$-monochromatic coloring of
$\mathcal{B}_{i}$. The same holds for colorings $\bar{R}(1,2,3)\mathrm{,}$
$R_{I}(1,2,3)\mathrm{,}$ and $\bar{R}_{I}(1,2,3)\mathrm{.}$ In general, the
following observation holds.

\begin{remark}
\label{Obs_rightR}Let $G$ be a snark with a normal $5$-edge-coloring $\sigma$,
$C$ a cycle in $G,$ and $\tilde{G}=G_{C}(\mathcal{A},\mathcal{B})$ a
superposition of $G$. Colorings $R$, $\bar{R}\mathrm{,}$ $R_{I}$ and $\bar
{R}_{I}$ of $\mathcal{B}_{i}$ are all right-side $\sigma$-monochromatic.
\end{remark}

We will need one more similar definition. A normal $5$-edge-coloring
$\tilde{\sigma}_{i}$ of $\mathcal{B}_{i}$ is \emph{left-side }$\sigma
$\emph{-compatible} if the following holds:

\begin{itemize}
\item $\tilde{\sigma}_{i}[s_{d_{i}}^{l}]\approx(\sigma(e_{i}),\{\sigma
(e_{i-1}),\sigma(f_{i})\})$,

\item $\tilde{\sigma}_{i}[s_{j}^{l}]\approx(\sigma(e_{i-1}),\{\sigma
(e_{i}),\sigma(f_{i})\})$ for every $j\in\{1,2,3\}\backslash\{d_{i}\},$ and

\item there exists a Kempe $(\sigma(e_{i-1}),\sigma(f_{i}))$-chain $P^{l}$
connecting a pair of left semiedges $s_{j}^{l}$ for $j\not =d_{i}.$
\end{itemize}

\noindent For example, the coloring $R(1,2,3)$ from Figure
\ref{Fig_rightColoring} is left-side $\sigma$-compatible coloring of
$\mathcal{B}_{i}$ if $d_{i}=2$ and $\sigma$ is as in Figure
\ref{Fig_coloringCycle}, and for $d_{i}=3$ the same holds for the coloring
$R_{I}(1,2,3)\mathrm{.}$ In general, the following observation holds.

\begin{remark}
\label{Obs_leftR}Let $G$ be a snark with a normal $5$-edge-coloring $\sigma$,
$C$ a cycle in $G,$ and $\tilde{G}=G_{C}(\mathcal{A},\mathcal{B})$ a
superposition of $G$. If $d_{i}=2$ (resp. $d_{i}=3$), then a $\sigma$-colored
coloring $R$ (resp. $R_{I}$) of $\mathcal{B}_{i}$ is left-side $\sigma$-compatible.
\end{remark}

\begin{figure}[h]
\begin{center}%
\begin{tabular}
[t]{ll}%
a) & \raisebox{-0.9\height}{\includegraphics[scale=0.6]{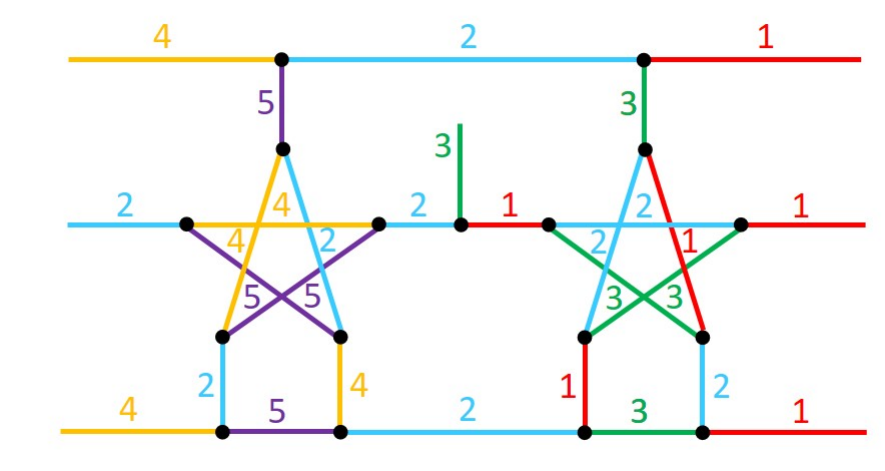}}\\
b) & \raisebox{-0.9\height}{\includegraphics[scale=0.6]{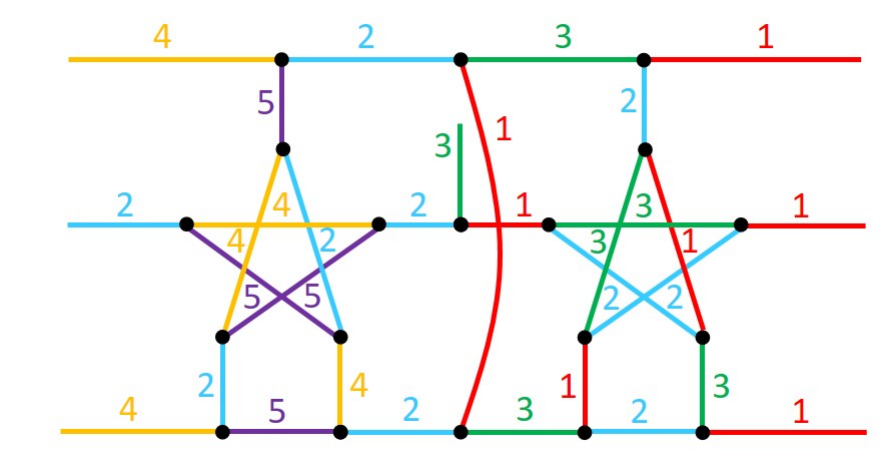}}
\end{tabular}
\end{center}
\caption{For a coloring $\sigma$ of $G$ as in Figure \ref{Fig_coloringCycle},
this figure shows $\sigma$-colored colorings: a) $\tilde{\sigma}_{i-1}=R$ and
$\tilde{\sigma}_{i}=R$ in case of $\mathcal{A}_{i}=A$; b) $\tilde{\sigma
}_{i-1}=R$ and $\tilde{\sigma}_{i}^{\prime}=\bar{R}$ in case of $\mathcal{A}%
_{i}=A^{\prime}.$ }%
\label{Fig_connectionRightLeft}%
\end{figure}

Our aim is to extend a normal $5$-edge-coloring $\sigma$ of a snark $G$ to its
superposition $\tilde{G}=G_{C}(\mathcal{A},\mathcal{B})$ by preserving colors
$\sigma(e)$ for edges $e$ of $G$ outside of the cycle $C.$ For that purpose we
define an induced submultipole $M_{\mathrm{int}}$ of $G$ as the submultipole
induced by the set of vertices $V(G)\backslash V(C),$ i.e. we define
$M_{\mathrm{int}}=G[V(G)\backslash V(C)].$ Notice that $M_{\mathrm{int}}$ is
also an induced submultipole of $\tilde{G}.$ For a normal $5$-edge-coloring
$\sigma$ of $G,$ we also define $\tilde{\sigma}_{\mathrm{int}}$ as the
restriction of $\sigma$ to the submultipole $M_{\mathrm{int}}$ of $G,$ i.e.
$\tilde{\sigma}_{\mathrm{int}}=\left.  \sigma\right\vert _{M_{\mathrm{int}}}.$
In order to construct a normal $5$-edge-coloring $\tilde{\sigma}$ of the
entire $\tilde{G},$ we consider the coloring $\tilde{\sigma}_{\mathrm{int}}$
together with normal $5$-edge-colorings $\tilde{\sigma}_{i}$ of all superedges
$\mathcal{B}_{i},$ and we wish $\tilde{\sigma}_{\mathrm{int}},$ $\tilde
{\sigma}_{i-1}$ and $\tilde{\sigma}_{i}$ to be compatible in $\tilde{G}$ for
every $i.$ If we manage to find such colorings $\tilde{\sigma}_{i},$ then the
construction of $\tilde{\sigma}$ is successfully finished. Notice that the
colorings $\tilde{\sigma}_{i-1},$ $\tilde{\sigma}_{i}$ and $\tilde{\sigma
}_{\mathrm{int}}$ also induce the coloring of $\mathcal{A}_{i},$ so these
three colorings are compatible in $\tilde{G}$ if there exists a coloring
$\tilde{\sigma}$ of a submultipole of $\tilde{G}$ induced by
$V(M_{\mathrm{int}})\cup V(\mathcal{B}_{i-1})\cup V\mathcal{(\mathcal{A}}%
_{i})\cup V(\mathcal{B}_{i})$ such that $\left.  \tilde{\sigma}\right\vert
_{M_{\mathrm{int}}}=\tilde{\sigma}_{\mathrm{int}}$, $\left.  \tilde{\sigma
}\right\vert _{\mathcal{B}_{i-1}}=\tilde{\sigma}_{i-1}$ and $\left.
\tilde{\sigma}\right\vert _{\mathcal{\mathcal{B}}_{i}}=\tilde{\sigma}_{i}.$
The following lemma is the main tool in constructing a normal $5$%
-edge-coloring of a superposition $\tilde{G}$ in our approach.

\begin{lemma}
\label{Lemma_borderCompatibility}Let $G$ be a snark with a normal
$5$-edge-coloring $\sigma$ and $C$ a cycle in $G.$ Let $\tilde{G}%
=G_{C}(\mathcal{A},\mathcal{B})$ be a superposition of $G$, $\tilde{\sigma
}_{i-1}$ a right-side $\sigma$-monochromatic coloring of $\mathcal{B}_{i-1}$,
and $\tilde{\sigma}_{i}$ a left-side $\sigma$-compatible coloring of
$\mathcal{B}_{i}.$

\begin{itemize}
\item If $\mathcal{A}_{i}=A,$ then $\tilde{\sigma}_{i-1},$ $\tilde{\sigma}%
_{i}$ and $\tilde{\sigma}_{\mathrm{int}}$ are compatible in $\tilde{G}.$

\item If $\mathcal{A}_{i}=A^{\prime},$ then then $\tilde{\sigma}_{i-1},$
$\tilde{\sigma}_{i}^{\prime}$ and $\tilde{\sigma}_{\mathrm{int}}$ are
compatible in $\tilde{G},$ where $\tilde{\sigma}_{i}^{\prime}$ denotes a
coloring obtained from $\tilde{\sigma}_{i}$ by swapping colors along the Kempe
chain $P^{l}.$
\end{itemize}
\end{lemma}

\begin{proof}
Since $\mathcal{B}_{i-1}$ is right-side $\sigma$-monochromatic, we have
\[
\tilde{\sigma}_{i-1}[s_{1}^{r}]\approx\tilde{\sigma}_{i-1}[s_{2}^{r}%
]\approx\tilde{\sigma}_{i-1}[s_{3}^{r}]\approx(\sigma(e_{i-1}),\{\sigma
(e_{i-2}),\sigma(f_{i-1})\}).
\]
The choice of colors on the right side to be monochromatic implies that if the
three right semiedges can connect to left superedges of $\mathcal{B}_{i}$ for
one permutation $p_{i-1},$ then they can connect to them for all permutations
$p_{i-1}$. Thus, without the loss of generality we may assume $p_{i-1}%
=(1,2,3),$ and the question is whether for such a convenient choice of the
colors on the right side there exists a coloring of the entire superedge such
that the left side connects to $\mathcal{B}_{i-2}$ for any possible color
combinations $\sigma(u_{i-1}).$

The proof of this lemma, for $p_{i-1}=(1,2,3)$ and $d_{i}=2,$ is illustrated
by Figure \ref{Fig_connectionRightLeft}. Assume that the vertex of
$\mathcal{B}_{i}$ incident to the left semiedge $s_{j}^{l}$ (resp. right
semiedge $s_{j}^{r}$) is denoted by $w_{j}$ (resp. $z_{j}$) as in Figure
\ref{Fig_PetersenSuperedge}. Obviously, $w_{1}=z_{1}.$

We will slightely abuse the notation, and use the same notation $w_{j}$ and
$z_{j}$ in $\mathcal{B}_{i},$ $\mathcal{B}_{i-1}$ and $\mathcal{B}_{i-2},$ as
we usualy mention these vertices in the context of $\tilde{\sigma}_{i}%
(w_{j}),$ $\tilde{\sigma}_{i-1}(w_{j})$ and $\tilde{\sigma}_{i-}(w_{j}),$ to
which superedge we refer should be clear from the subscript of $\tilde{\sigma
}.$ Now, we distinguish the following two cases.

\bigskip\noindent\textbf{Case 1:} $\mathcal{A}_{i}=A.$ Assume that for a
coloring $\tilde{\sigma}$ it holds that $\left.  \tilde{\sigma}\right\vert
_{\mathcal{B}_{i-1}}=\tilde{\sigma}_{i-1},$ $\left.  \tilde{\sigma}\right\vert
_{\mathcal{B}_{i}}=\tilde{\sigma}_{i}$ and $\left.  \tilde{\sigma}\right\vert
_{M_{\mathrm{int}}}=\tilde{\sigma}_{\mathrm{int}}=\left.  \sigma\right\vert
_{M_{\mathrm{int}}}.$ We have to establish that $\tilde{\sigma}$ is well
defined and normal. First, notice that
\[
\tilde{\sigma}(u_{i})=\{\tilde{\sigma}_{i-1}(s_{d_{i}}^{r}),\tilde{\sigma}%
_{i}(s_{d_{i}}^{l}),\tilde{\sigma}_{\mathrm{int}}(f_{i})\}=\{\sigma
(e_{i-1}),\sigma(e_{i}),\sigma(f_{i})\}=\sigma(u_{i}),
\]
so $\tilde{\sigma}$ is proper at $u_{i}$ since $\sigma$ is proper at $u_{i}.$
This also implies that $f_{i}$ is normal by $\tilde{\sigma}$ since it is
normal by $\sigma$ in $G.$ Also, $\tilde{\sigma}_{i-1}(z_{d_{i}}%
)=\{\sigma(e_{i-2}),\sigma(e_{i-1}),\sigma(f_{i-1})\}$ further implies that
$z_{d_{i}}u_{i}$ is normal by $\tilde{\sigma}$ since $e_{i-1}=u_{i-1}u_{i}$ is
normal by $\sigma$ in $G,$ and $\tilde{\sigma}_{i}(w_{d_{i}})=\{\sigma
(e_{i-1}),\sigma(e_{i}),\sigma(f_{i})\}$ further implies that $z_{d_{i}}u_{i}$
is normal by $\tilde{\sigma}$ since $e_{i}=u_{i}u_{i+1}$ is normal by $\sigma$
in $G.$

It remains to consider each edge $z_{j}w_{j}$ for $j\not =d_{i}.$ Notice that
\[
\tilde{\sigma}(z_{j}w_{j})=\tilde{\sigma}_{i-1}(s_{j}^{r})=\tilde{\sigma}%
_{i}(s_{j}^{l})=\sigma(e_{i-1}),
\]
so $\tilde{\sigma}$ is well defined. Further, from $\tilde{\sigma}_{i-1}%
(z_{j})=\{\sigma(e_{i-2}),\sigma(e_{i-1}),\sigma(f_{i-1})\}$ and
$\tilde{\sigma}_{i}(w_{j})=\{\sigma(e_{i-1}),\sigma(e_{i}),\sigma(f_{i})\}$ it
follows that $z_{j}w_{j}$ is normal by $\tilde{\sigma}$, since $e_{i-1}$ is
normal by $\sigma$ in $G.$

\bigskip\noindent\textbf{Case 2:} $\mathcal{A}_{i}=A^{\prime}.$ Assume that
for coloring $\tilde{\sigma}$ it holds that $\left.  \tilde{\sigma}\right\vert
_{\mathcal{B}_{i-1}}=\tilde{\sigma}_{i-1},$ $\left.  \tilde{\sigma}\right\vert
_{\mathcal{B}_{i}}=\tilde{\sigma}_{i}^{\prime}$ and $\left.  \tilde{\sigma
}\right\vert _{M_{\mathrm{int}}}=\tilde{\sigma}_{\mathrm{int}}=\left.
\sigma\right\vert _{M_{\mathrm{int}}}.$ We also define $\tilde{\sigma}%
(u_{i}^{\prime}u_{i}^{\prime\prime})=\sigma(e_{i}).$ Since for $j=d_{i}$ it
holds that $\tilde{\sigma}_{i}^{\prime}[s_{d_{i}}^{l}]\approx\tilde{\sigma
}_{i}[s_{d_{i}}^{l}],$ the properness of $\tilde{\sigma}$ at $u_{i}$ and the
normality of all edges incident to $u_{i}$ in $\tilde{G}$ follows as in Case
1. It remains to establish the properness of $\tilde{\sigma}$ at
$u_{i}^{\prime}$ and $u_{i}^{\prime\prime}$ and the normality of all edges
incident to them. For the sake of notation consistency, we denote
$u_{i}^{\prime}=u_{i}^{2}$ and $u_{i}^{\prime\prime}=u_{i}^{3}$ so that
$u_{i}^{j}$ is adjacent to $w_{j}.$ Notice that
\[
\tilde{\sigma}(u_{i}^{j})=\{\tilde{\sigma}_{i-1}(s_{j}^{r}),\tilde{\sigma}%
_{i}^{\prime}(s_{j}^{l}),\tilde{\sigma}(u_{i}^{\prime}u_{i}^{\prime\prime
})\}=\{\sigma(e_{i-1}),\sigma(f_{i}),\sigma(e_{i})\}=\sigma(u_{i}),
\]
thus $\tilde{\sigma}$ is proper at $u_{i}^{j}$ since $\sigma$ is proper at
$u_{i},$ and this holds for every $j=2,3.$

It remains to consider normality of the edges $u_{i}^{\prime}u_{i}%
^{\prime\prime},$ $u_{i}^{j}w_{j}$ and $u_{i}^{j}z_{j}$ for each $j=2,3.$ From
$\tilde{\sigma}(u_{i}^{\prime})=\tilde{\sigma}(u_{i}^{\prime\prime}%
)=\sigma(u_{i})$ it follows that $u_{i}^{\prime}u_{i}^{\prime\prime}$ is poor
by $\tilde{\sigma}.$ From $\tilde{\sigma}_{i}(w_{j})=\tilde{\sigma}(u_{i}%
^{j})=\sigma(u_{i})$ it follows that $u_{i}^{j}w_{j}$ is also poor by
$\tilde{\sigma}.$ Finally, $\tilde{\sigma}_{i-1}(z_{j})=\sigma(u_{i-1})$
together with $\tilde{\sigma}(u_{i}^{j})=\sigma(u_{i})$ implies that
$u_{i}^{j}z_{j}$ is normal by $\tilde{\sigma}$ since $e_{i-1}$ is normal by
$\sigma.$
\end{proof}

\begin{figure}[h]
\begin{center}%
\begin{tabular}
[t]{ll}%
a) & \raisebox{-0.9\height}{\includegraphics[scale=0.6]{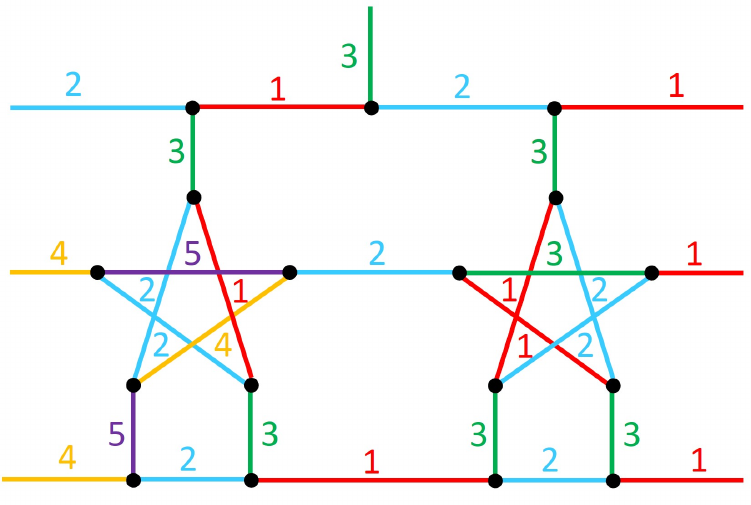}}\\
b) & \raisebox{-0.9\height}{\includegraphics[scale=0.6]{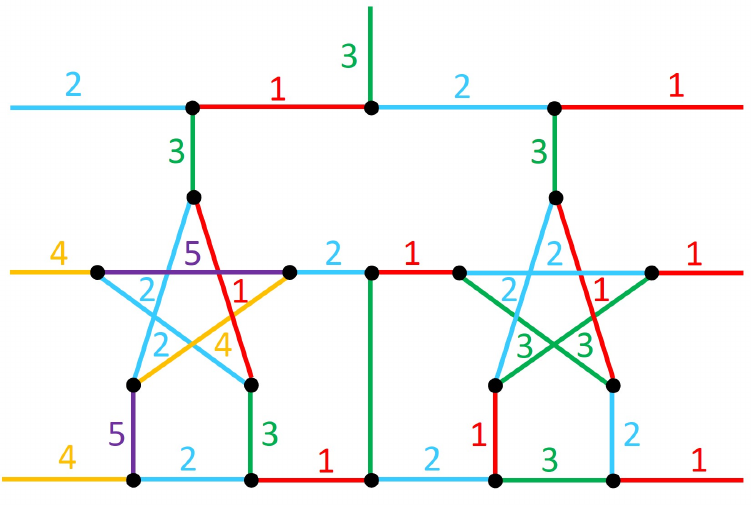}}
\end{tabular}
\end{center}
\caption{For a coloring $\sigma$ of $G$ as in Figure \ref{Fig_coloringCycle},
this figure shows $\sigma$-colored colorings $\tilde{\sigma}_{i-1}$ which is
left-side $\sigma$-compatible and $\tilde{\sigma}_{i}$ which is right-side
$\sigma$-monochromatic in case when $d_{i-1}=d_{i}=1$ and a) $\mathcal{A}%
_{i}=A,$ b) $\mathcal{A}_{i}=A^{\prime}.$}%
\label{Fig10}%
\end{figure}

The following result is established in \cite{SedSkrePaper1}. We give a sketch
of the proof here for the purpose of this paper being self-contained.

\begin{lemma}
\label{Lemma_parJedinica}Let $G$ be a snark with a normal $5$-edge-coloring
$\sigma$, $C$ a cycle in $G,$ and $\tilde{G}=G_{C}(\mathcal{A},\mathcal{B})$ a
superposition of $G$. If $\mathcal{B}_{i-1}$ and $\mathcal{B}_{i}$ are two
consecutive superedges with $d_{i-1}=d_{i}=1,$ then there exist a left-side
$\sigma$-compatible coloring $\tilde{\sigma}_{i-1}$ of $\mathcal{B}_{i-1}$ and
a right-side $\sigma$-monochromatic coloring $\tilde{\sigma}_{i}$ of
$\mathcal{B}_{i}$ such that $\tilde{\sigma}_{i-1},$ $\tilde{\sigma}_{i}$ and
$\tilde{\sigma}_{\mathrm{int}}$ are compatible in $\tilde{G}.$
\end{lemma}

\begin{proof}
Assume first that $\sigma$ is as in Figure \ref{Fig_coloringCycle}, then
$\tilde{\sigma}_{i-1}$ and $\tilde{\sigma}_{i}$ are as in Figure \ref{Fig10}
from which it is also obvious that they have the desired properties and are
compatible. Notice that $\tilde{\sigma}_{i}=R_{I}$ (resp. $\tilde{\sigma}%
_{i}=R$) in case of $\mathcal{A}_{i}=A$ (resp. $\mathcal{A}_{i}=A^{\prime}$).
Further, if $(\sigma(e_{i-2}),\sigma(f_{i-1}))$ take values from
$\{(5,4),(1,3),(3,1)\}$ instead, then the corresponding colorings are obtained
from these on Figure \ref{Fig10} by the corresponding replacement of colors
along the Kempe chain $P^{l}$ in $\tilde{\sigma}_{i-1}.$ Thus, we have covered
four combinations of colors on the edges incident to $u_{i-1}$ and $u_{i}$ in
$G.$ For all the remaining possibilities of colors on those five edges,
colorings $\tilde{\sigma}_{i-1}$ and $\tilde{\sigma}_{i}$ are obtained from
one of these four cases by the adequate color permutation.
\end{proof}

We will also need a result similar to Lemma \ref{Lemma_parJedinica}, when
$d_{i-1}=1$ and $d_{i}\not =1.$ To establish such a result, in the next
paragraph we introduce two more normal $5$-edge-colorings of the superedge
$(P_{10})_{u,v}.$

\begin{figure}[h]
\begin{center}%
\begin{tabular}
[t]{llll}%
a) & \raisebox{-0.9\height}{\includegraphics[scale=0.6]{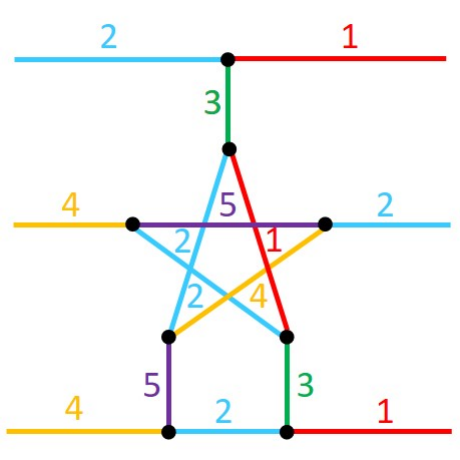}} & b) &
\raisebox{-0.9\height}{\includegraphics[scale=0.6]{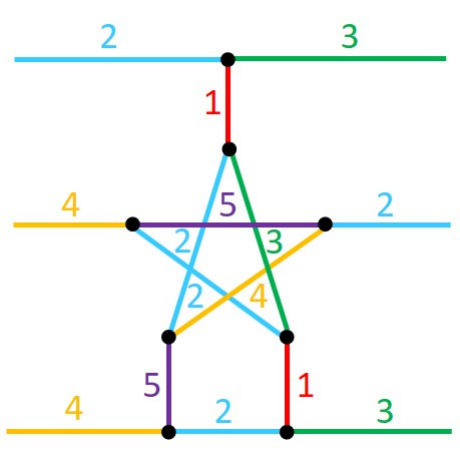}}
\end{tabular}
\end{center}
\caption{Normal $5$-edge colorings: a) $L^{(1)}(1,2,3,(4,5))$, b)
$L^{(2)}(1,2,3,(4,5))$, of the superedge $(P_{10})_{u,v}.$}%
\label{Fig_leftColorings}%
\end{figure}

The normal $5$-edge colorings $L^{(1)}(1,2,3,(4,5))$ and $L^{(2)}%
(1,2,3,(4,5))$ of $(P_{10})_{u,v}$ are defined as in Figure
\ref{Fig_leftColorings}. Notice that these two colorings are similar, i.e.
swapping colors $1$ and $3$ in $L^{(1)}(1,2,3,(4,5))$ results with
$L^{(2)}(1,2,3,(4,5)).$ The motivation for the notation is the following,
these two colorings use only colors $1,2,3,$ except along the Kempe
$(4,5)$-chain $P^{l}$ connecting left semiedges $s_{2}^{l}$ and $s_{3}^{l}.$
Obviously, these two colorings will remain normal $5$-edge colorings if colors
$(4,5)$ along the Kempe chain $P^{l}$ are replaced by $(t_{1},t_{2}%
)\in\{(1,3),(3,1),(5,4)\}.$ Thus obtained colorings of $(P_{10})_{u,v}$ will
be denoted by $L^{(1)}(1,2,3,(t_{1},t_{2}))$ and $L^{(2)}(1,2,3,(t_{1}%
,t_{2})).$

In order to obtain full generality, assume that $c$ is a permutation of the
set of five colors $\{1,2,3,4,5\}$ such that $c(i)=c_{i}$ and $c(t_{i}%
)=c_{i}^{\prime}.$ Applying the color permutation $c$ to a coloring
$L^{(k)}(1,2,3,(t_{1},t_{2}))$ yields a normal $5$-edge-coloring denoted by
$L^{(k)}(c_{1},c_{2},c_{3},(c_{1}^{\prime},c_{2}^{\prime}))$ for
$k\in\{1,2\}.$ A \emph{complementary} coloring of $L^{(k)}(c_{1},c_{2}%
,c_{3},(c_{1}^{\prime},c_{2}^{\prime}))$ is obtained from it by swapping
colors $c_{1}^{\prime}$ and $c_{2}^{\prime}$ along the Kempe chain $P^{l}$ and
it is denoted by $\bar{L}^{(k)}(c_{1},c_{2},c_{3},(c_{1}^{\prime}%
,c_{2}^{\prime})).$ Finally, applying the isomorphism $I$ to $L^{(k)}%
(c_{1},c_{2},c_{3},(c_{1}^{\prime},c_{2}^{\prime}))$ or its complementary
coloring $\bar{L}^{(k)}(c_{1},c_{2},c_{3},(c_{1}^{\prime},c_{2}^{\prime}))$
yields a normal $5$-edge-coloring of $(P_{10})_{u,v}$ denoted by $L_{I}%
^{(k)}(c_{1},c_{2},c_{3},(c_{1}^{\prime},c_{2}^{\prime}))$ and $\bar{L}%
_{I}^{(k)}(c_{1},c_{2},c_{3},(c_{1}^{\prime},c_{2}^{\prime})),$ respectively.

Notice that colorings $L^{(k)}$ and its variants $\bar{L}^{(k)},$ $L_{I}%
^{(k)},$ $\bar{L}_{I}^{(k)}$ are defined for any color permutation $c.$ Which
colors will be used in a coloring of a superedge $\mathcal{B}_{i}%
=(P_{10})_{u,v}$ within a superposition $\tilde{G}=G_{C}(\mathcal{A}%
,\mathcal{B})$ again depends on the coloring $\sigma$ of the edges of $G.$
Namely, let $G$ be a snark with a normal $5$-edge-coloring $\sigma,$ $C$ a
cycle in $G$ and $\tilde{G}=G_{C}(\mathcal{A},\mathcal{B})$ a superposition of
$G.$ A coloring $L^{(k)}(c_{1},c_{2},c_{3},(c_{1}^{\prime},c_{2}^{\prime}))$
is said to be a $\sigma$\emph{-colored} coloring of $\mathcal{B}_{i}$ if
$c_{1}=\sigma(e_{i+1}),$ $c_{2}=\sigma(e_{i}),$ $c_{3}=\sigma(f_{i}),$
$c_{1}^{\prime}=\sigma(e_{i-1})$, $c_{2}^{\prime}=\sigma(f_{i-1}),$ in which
case it is denoted by $L^{(k)}.$ Similarly, we define $\sigma$\emph{-colored}
$\bar{L}^{(k)},$ $L_{I}^{(k)}$ and $\bar{L}_{I}^{(k)}.$ For example, if
$\sigma$ is as in Figure \ref{Fig_coloringCycle}, then both colorings
$L^{(1)}(1,2,3,(4,5))$ and $L^{(2)}(1,2,3,(4,5))$ from Figure
\ref{Fig_leftColorings} are $\sigma$-colored colorings of $\mathcal{B}_{i}$.
Moreover, for such a coloring $\sigma$ those two colorings are also left-side
$\sigma$-compatible. In general, the following observation holds.

\begin{remark}
\label{Obs_leftL}Let $G$ be a snark with a normal $5$-edge-coloring $\sigma$,
$C$ a cycle in $G,$ and $\tilde{G}=G_{C}(\mathcal{A},\mathcal{B})$ a
superposition of $G$. If $d_{i}=1$ for some $i\in\{0,\ldots,g-1\}$, then each
of the colorings $L^{(k)},$ $\bar{L}^{(k)},$ $L_{I}^{(k)}$ and $\bar{L}%
_{I}^{(k)}$ of $\mathcal{B}_{i}$ is left-side $\sigma$-compatible for every
$k\in\{1,2\}.$
\end{remark}

\begin{figure}[ptbh]
\begin{center}%
\begin{tabular}
[t]{lll}%
a) & \raisebox{-0.9\height}{\includegraphics[scale=0.5]{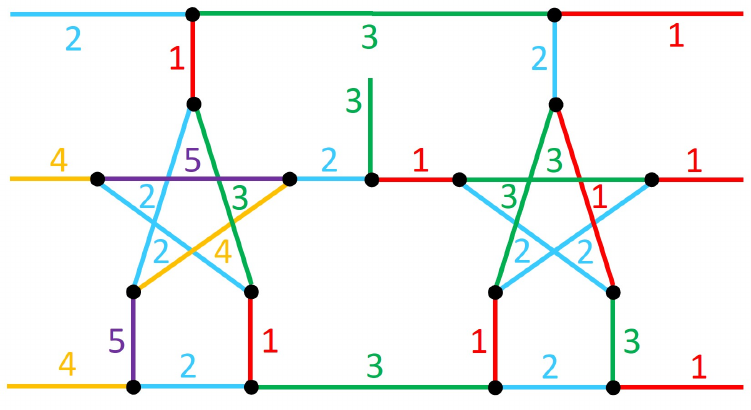}} &
\raisebox{-0.9\height}{\includegraphics[scale=0.5]{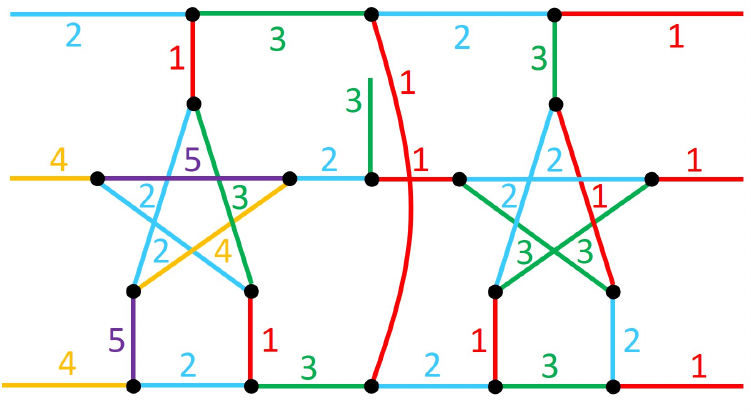}}\\
b) & \raisebox{-0.9\height}{\includegraphics[scale=0.5]{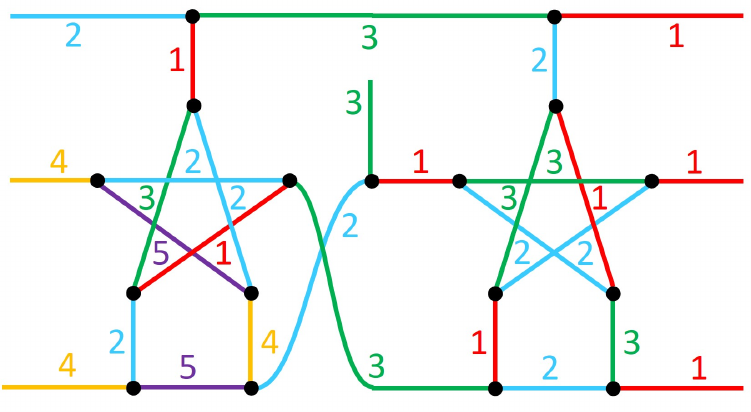}} &
\raisebox{-0.9\height}{\includegraphics[scale=0.5]{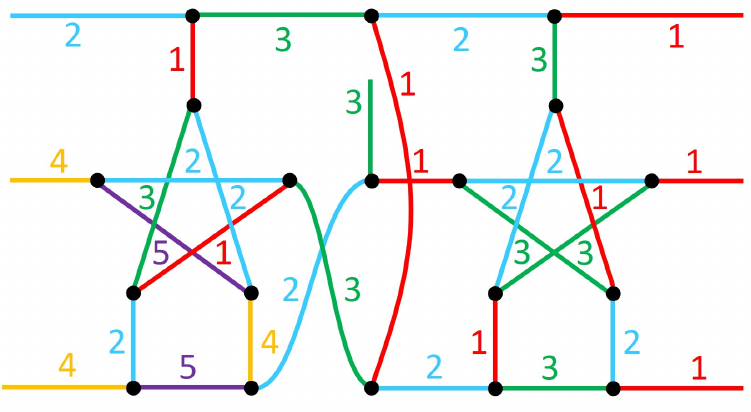}}\\
c) & \raisebox{-0.9\height}{\includegraphics[scale=0.5]{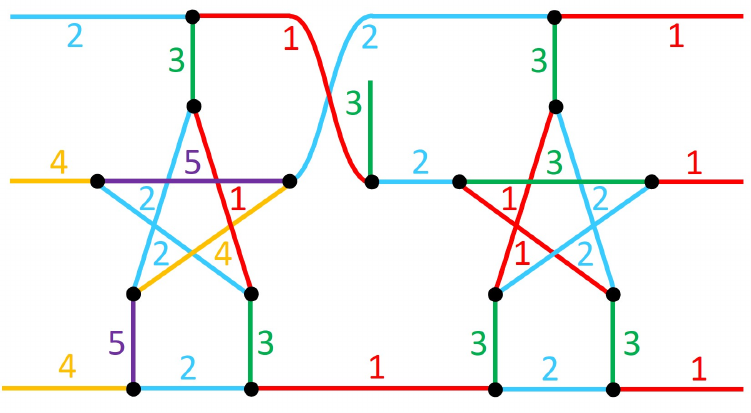}} &
\raisebox{-0.9\height}{\includegraphics[scale=0.5]{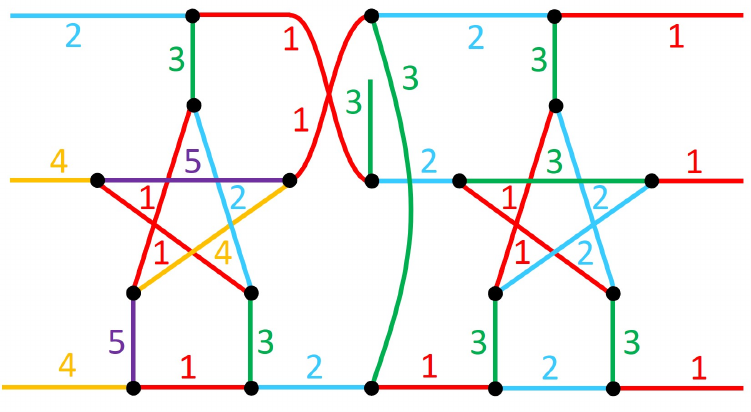}}\\
d) & \raisebox{-0.9\height}{\includegraphics[scale=0.5]{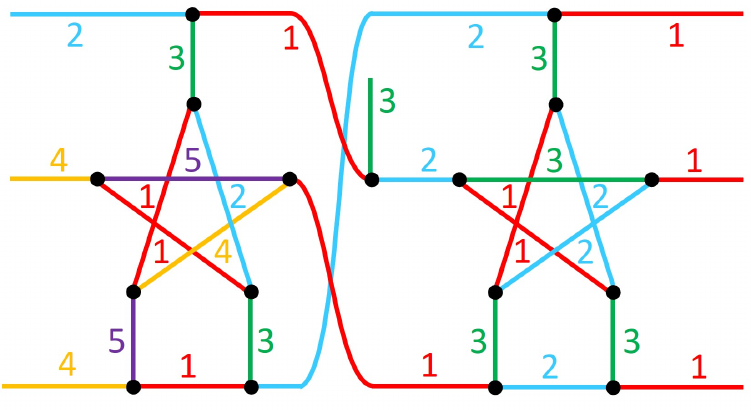}} &
\raisebox{-0.9\height}{\includegraphics[scale=0.5]{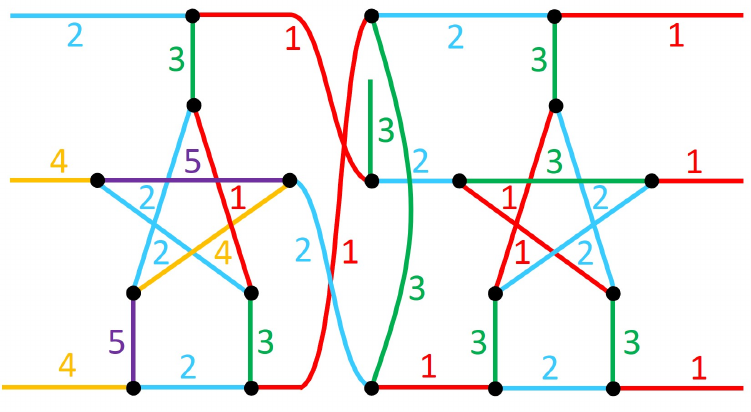}}\\
e) & \raisebox{-0.9\height}{\includegraphics[scale=0.5]{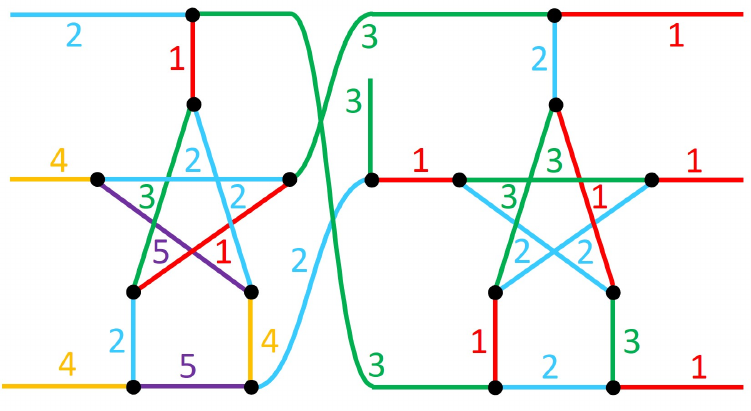}} &
\raisebox{-0.9\height}{\includegraphics[scale=0.5]{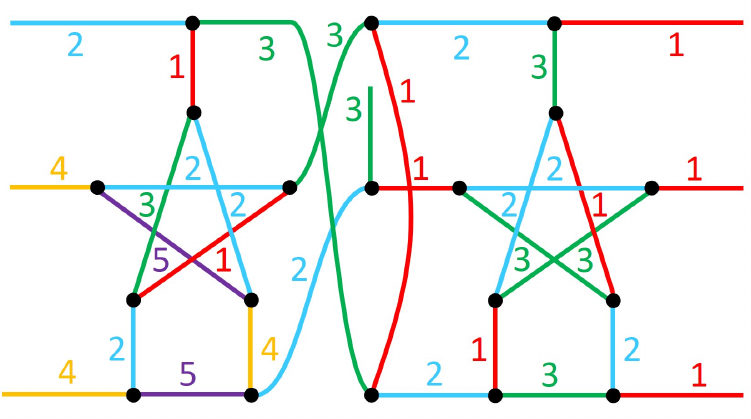}}\\
f) & \raisebox{-0.9\height}{\includegraphics[scale=0.5]{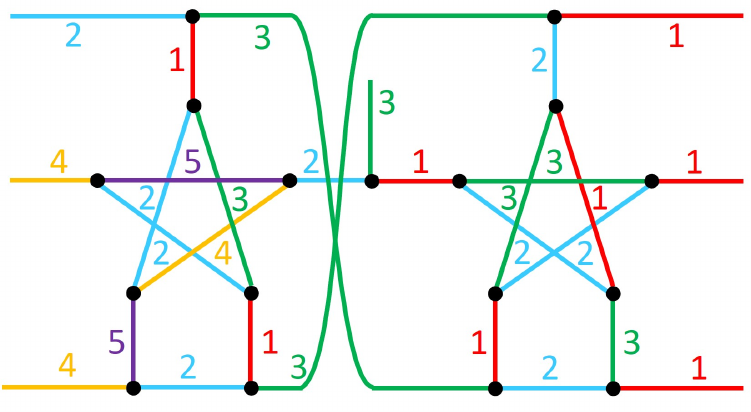}} &
\raisebox{-0.9\height}{\includegraphics[scale=0.5]{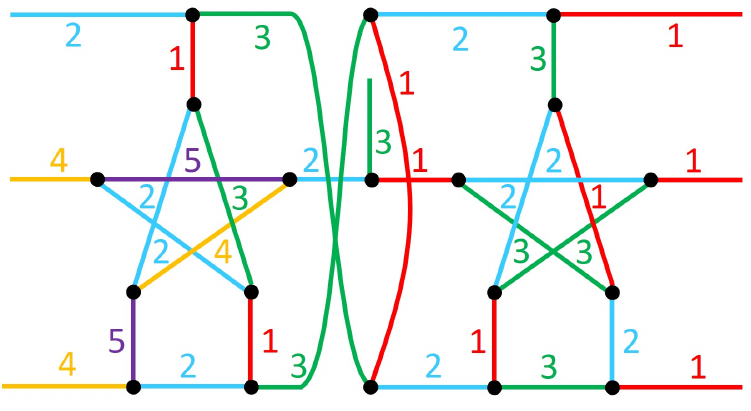}}
\end{tabular}
\end{center}
\caption{For a normal $5$-edge-coloring $\sigma$ of $G$ from Figure
\ref{Fig_coloringCycle}, this figure shows a normal $5$-edge colorings of
$\mathcal{B}_{i-1}$ and $\mathcal{B}_{i}$ compatible with $\sigma$ when
$\mathcal{A}_{i}=A$ (on the left) and $\mathcal{A}_{i}=A^{\prime}$ (on the
right) for $d_{i}=2$ and: a) $p_{i-1}=(1,2,3),$ b) $p_{i-1}=(1,3,2),$ c)
$p_{i-1}=(2,1,3),$ d) $p_{i-1}=(2,3,1),$ e) $p_{i-1}=(3,1,2),$ f)
$p_{i-1}=(3,2,1)$.}%
\label{Fig_jednaJedinica}%
\end{figure}

We now need the following lemma, which complements Lemma
\ref{Lemma_parJedinica}.

\begin{lemma}
\label{Lemma_jednaJedinica}Let $G$ be a snark with a normal $5$-edge-coloring
$\sigma$, $C$ a cycle in $G,$ and $\tilde{G}=G_{C}(\mathcal{A},\mathcal{B})$ a
superposition of $G$. Suppose that $\mathcal{B}_{i-1}$ and $\mathcal{B}_{i}$
are two consecutive superedges such that $d_{i-1}=1$ and $d_{i}\not =1.$ Then,
there exist a left-side $\sigma$-compatible coloring $\tilde{\sigma}_{i-1}$ of
$\mathcal{B}_{i-1}$ and a right-side $\sigma$-monochromatic coloring
$\tilde{\sigma}_{i}$ of $\mathcal{B}_{i}$ such that $\tilde{\sigma}_{i-1},$
$\tilde{\sigma}_{i}$ and $\tilde{\sigma}_{\mathrm{int}}$ are compatible in
$G.$
\end{lemma}

\begin{proof}
Assume first that $d_{i}=2.$ We define $\tilde{\sigma}_{i-1}$ and
$\tilde{\sigma}_{i}$ according to the following table
\[%
\begin{tabular}
[c]{|c||c|c||c|c|c|}\hline
& \multicolumn{2}{||c||}{$\mathcal{A}_{i}=A$} &
\multicolumn{3}{||c|}{$\mathcal{A}_{i}=A^{\prime}$}\\\hline
& $\tilde{\sigma}_{i-1}$ & $\tilde{\sigma}_{i}$ & $\tilde{\sigma}_{i-1}$ &
$\tilde{\sigma}_{i}$ & $\tilde{\sigma}(u_{i}^{\prime}u_{i}^{\prime\prime}%
)$\\\hline\hline
$p_{i-1}=(1,2,3)$ & $L^{(2)}$ & $\overset{}{\bar{R}}$ & $L^{(2)}$ & $R$ &
$\sigma(e_{i})$\\\hline
$p_{i-1}=(1,3,2)$ & $L_{I}^{(2)}$ & $\bar{R}$ & $L_{I}^{(2)}$ & $R$ &
$\sigma(e_{i})$\\\hline
$p_{i-1}=(2,1,3)$ & $L^{(1)}$ & $R_{I}$ & $L_{I}^{(1)}$ & $R_{I}$ &
$\sigma(f_{i})$\\\hline
$p_{i-1}=(2,3,1)$ & $L_{I}^{(1)}$ & $R_{I}$ & $L^{(1)}$ & $R_{I}$ &
$\sigma(f_{i})$\\\hline
$p_{i-1}=(3,1,2)$ & $L_{I}^{(2)}$ & $\bar{R}$ & $L_{I}^{(2)}$ & $R$ &
$\sigma(e_{i})$\\\hline
$p_{i-1}=(3,2,1)$ & $L^{(2)}$ & $\bar{R}$ & $L^{(2)}$ & $R$ & $\sigma(e_{i}%
)$\\\hline
\end{tabular}
\ \ \ \ \ \
\]
To see that $\tilde{\sigma}_{i-1},$ $\tilde{\sigma}_{i}$ and $\tilde{\sigma
}_{\mathrm{int}}$ defined as in the above table are indeed compatible in
$\tilde{G},$ assume first $(\sigma(e_{i}),\sigma(e_{i-1}),\sigma
(f_{i}))=(1,2,3).$ If also $(\sigma(e_{i-2}),\sigma(f_{i-1}))=(4,5),$ then it
is enough to see Figure \ref{Fig_jednaJedinica}. Notice that $(\sigma
(e_{i-2}),\sigma(f_{i-1}))$ can also take values $(t_{1},t_{2}%
)=\{(1,3),(3,1),(5,4)\}$ in which case colors $(4,5)$ in a Kempe chain $P^{l}$
of $\tilde{\sigma}_{i-1}$ are replaced by $(t_{1},t_{2})$ and the desired
colorings are obtained. For any other values of $\sigma$ on the edges incident
to $u_{i-1}$ and $u_{i}$ the desired coloring can be obtained from one of
these four situations by an adequate color permutation.

If $d_{i}=3,$ the desired colorings $\tilde{\sigma}_{i-1}$ and $\tilde{\sigma
}_{i}$ are obtained by applying first the isomorphism $I$ to both
$\mathcal{B}_{i-1}$ and $\mathcal{B}_{i},$ then color them as in the previous
case, such colorings are denoted by $\tilde{\sigma}_{i-1}^{\prime}$ and
$\tilde{\sigma}_{i}^{\prime},$ respectively. Then apply again the isomorphism
$I$ to both $\mathcal{B}_{i-1}$ and $\mathcal{B}_{i}$ and their colorings,
i.e. $\tilde{\sigma}_{i-1}=(\tilde{\sigma}_{i-1}^{\prime})_{I}$ and
$\tilde{\sigma}_{i}=(\tilde{\sigma}_{i}^{\prime})_{I}.$
\end{proof}

\section{Main results}

We are now in a position to establish the main result of this paper. Let us
note that in \cite{SedSkrePaper1}, the following theorem is established.

\begin{theorem}
\label{tm_stari}Let $G$ be a snark, $\sigma$ a normal $5$-edge coloring of
$G,$ $C$ a cycle of length $g$ in $G,$ and $\tilde{G}\in\mathcal{G}%
_{C}(\mathcal{A},\mathcal{B})$ a superposition of $G.$ If $p_{i}=(1,2,3)$ and
$d_{i}=1$ for every $i\in\{0,\ldots,g-1\},$ then for even $g$ there exists a
normal $5$-edge coloring $\tilde{\sigma}$ of $\tilde{G}$ with at least $18$
poor edges.
\end{theorem}

Moreover, regarding cycles $C$ with an odd length $g,$ an example of the
superposition $p_{i}=(1,2,3)$ and $d_{i}=1$ for every $i\in\{0,\ldots,g-1\}$
is provided in \cite{SedSkrePaper1} such that the superposition $\tilde{G}$
cannot be colored by preserving the colors of $\sigma$ outside the cycle $C.$
Now, we can extend the results to all the remaining ways of semiedge
identification of superedges $\mathcal{B}_{i}.$

\begin{theorem}
\label{tm_main}Let $G$ be a snark, $\sigma$ a normal $5$-edge coloring of $G,$
$C$ a cycle of length $g$ in $G,$ and $\tilde{G}\in\mathcal{G}_{C}%
(\mathcal{A},\mathcal{B})$ a superposition of $G.$ If there exists at least
one $i\in\{0,\ldots,g-1\}$ such that $p_{i}(1)\not =1$ or $d_{i}\not =1$, then
there exists a normal $5$-edge coloring $\tilde{\sigma}$ of $\tilde{G}$ with
at least $18$ poor edges.
\end{theorem}

\begin{proof}
First notice the following, if $d_{i}=1$ for every $i\in\{0,\ldots,g-1\},$
then the assumption of the theorem implies $p_{i}(1)\not =1$ for some $i.$ In
that case, relabeling vertices of the cycle $C$ by $u_{i}=u_{g-1-i}^{\prime}$,
i.e. changing the direction of the cycle $C$, yields a superposition in which
$d_{i}\not =1$ for at least one $i.$ Thus, the assumption of the theorem
reduces to $d_{i}\not =1$ for some $i,$ say $d_{0}\not =1.$

We say that a superedge $\mathcal{B}_{i}$ is a \emph{start }superedge if
$d_{i}=1$ and $d_{i-1}\not =1,$ and it is an \emph{end }superedge if $d_{i}=1$
and $d_{i+1}\not =1,$ where indices are taken modulo $g.$ Let $\mathcal{B}%
_{i^{\ast}}$ be a start superedge for $i^{\ast}\in\{0,\ldots,g-1\},$ an end
superedge $\mathcal{B}_{i^{\ast}+r}$ \emph{corresponds} to the start superedge
$\mathcal{B}_{i^{\ast}}$ if $r\geq0$ is the smallest number such that
$\mathcal{B}_{i^{\ast}+r}$ is an end superedge, i.e. if there are no end
superedges between $\mathcal{B}_{i^{\ast}}$ and $\mathcal{B}_{i^{\ast}+r}.$
Notice that this also means there are no start superedges between a pair of
corresponding start and end superedges. Since we assumed $d_{0}\not =1,$ at
least one pair of corresponding start and end superedges in $\tilde{G}$ exists
if there is at least one superedge $\mathcal{B}_{i}$ in $\tilde{G}$ with
$d_{i}=1.$

Let $\mathcal{B}_{i^{\ast}}$ be a start superedge and $\mathcal{B}_{i^{\ast
}+r}$ its corresponding end superedge, an \emph{even-chain} $S_{r}%
(\mathcal{B}_{i^{\ast}})$ of the start superedge $\mathcal{B}_{i^{\ast}}$ is
defined as the sequence of consecutive superedges $\mathcal{B}_{i^{\ast}%
},\mathcal{B}_{i^{\ast}+1},\ldots,\mathcal{B}_{i^{\ast}+r}$ if $r$ is odd and
$\mathcal{B}_{i^{\ast}},\mathcal{B}_{i^{\ast}+1},\ldots,\mathcal{B}_{i^{\ast
}+r+1}$ if $r$ is even. Notice that the number of superedges in an even-chain
is always even. Notice also, in the case of odd $r$ all superedges in an
even-chain have their dock equal to $1,$ i.e. $d_{i^{\ast}+p}=1$ for every
$p\in\{0,\ldots,r\}.$ In case of even $r$ all superedges in an even-chain
except the last one have their dock equal to $1,$ while the last superedge
$\mathcal{B}_{i^{\ast}+r+1}$ has $d_{i^{\ast}+r+1}\not =1.$ Notice that
even-chains are disjoint, i.e. every superedge $\mathcal{B}_{i}$ of $\tilde
{G}$ belongs to at most one even-chain.

We would like to color each superedge by a $\sigma$-compatible coloring which
is both right-side $\sigma$-monochromatic and left-side $\sigma$-compatible.
This is not always possible, so in some cases we consider jointly two
consecutive superedges. Namely, we divide all $\mathcal{B}_{i}$'s as follows.
If $\mathcal{B}_{i}$ does not belong to any even-chain it is called a
\emph{singleton}. A pair of superedges $\mathcal{B}_{i^{\ast}+2p}$ and
$\mathcal{B}_{i^{\ast}+2p+1}$ which belong to an even-chain $S_{r}%
(\mathcal{B}_{i^{\ast}})$ is called a \emph{conjoined pair}, where
$\mathcal{B}_{i^{\ast}+2p}$ is the \emph{left} and $\mathcal{B}_{i^{\ast
}+2p+1}$ the \emph{right} superedge of the conjoined pair. Recall that the two
consecutive superedges of an even chain both have the dock equal to $1$,
except for the pair of the last two superedges in the even chain
$S_{r}(\mathcal{B}_{i^{\ast}})$ when $r$ is even. These pairs will have to be
treated separately, hence we say that a conjoined pair is \emph{odd }if $2p=r$
and $r$ is even, otherwise it is \emph{regular}.

Now, for every $i\in\{0,\ldots,g-1\}$ we define a coloring $\tilde{\sigma}%
_{i}$ of $\mathcal{B}_{i}$ as follows. If $\mathcal{B}_{i}$ is a singleton
then $\tilde{\sigma}_{i}$ is defined according to the following table%
\[%
\begin{tabular}
[c]{|l|l|l|}\hline
$\mathcal{A}_{i}$ & $d_{i}$ & $\tilde{\sigma}_{i}$\\\hline\hline
$A$ & $2$ & $R$\\\cline{2-3}
& $3$ & $R_{I}$\\\hline
$A^{\prime}$ & $2$ & $\overset{}{\bar{R}}$\\\cline{2-3}
& $3$ & $\overset{}{\bar{R}}_{I}$\\\hline
\end{tabular}
\ \ .
\]
Further, if $\mathcal{B}_{i-1}$ and $\mathcal{B}_{i}$ are a regular conjoined
pair, then $\tilde{\sigma}_{i-1}$ and $\tilde{\sigma}_{i}$ are defined as in
Lemma \ref{Lemma_parJedinica}. Finally, if $\mathcal{B}_{i-1}$ and
$\mathcal{B}_{i}$ are an odd conjoined pair, then $\tilde{\sigma}_{i-1}$ and
$\tilde{\sigma}_{i}$ are defined as in Lemma \ref{Lemma_jednaJedinica}.
Additionally, in the case of every conjoined pair $\mathcal{B}_{i-1}$ and
$\mathcal{B}_{i}$, if $\mathcal{A}_{i-1}=A^{\prime}$ the colors along the
Kempe chain $P^{l}$ in $\tilde{\sigma}_{i-1}$ are then further swapped, this
is necessary in order for the left connector of $\mathcal{B}_{i-1}$ to be
compatible over $\mathcal{A}_{i-1}=A^{\prime}$ to right connector of
$\mathcal{B}_{i-1}$ which is right-side $\sigma$-monochromatic.

It is now necessary to establish that $\tilde{\sigma}_{i-1}$, $\tilde{\sigma
}_{i}$ and $\tilde{\sigma}_{\mathrm{int}}$ are compatible in $\tilde{G}$ for
every $i\in\{0,\ldots,g-1\},$ i.e. that a coloring $\tilde{\sigma}$ of
$\tilde{G}$ such that $\left.  \tilde{\sigma}\right\vert _{M_{\mathrm{int}}%
}=\tilde{\sigma}_{\mathrm{int}}$ and $\left.  \tilde{\sigma}\right\vert
_{\mathcal{B}_{i}}=\tilde{\sigma}_{i}$ for every $i$ is a well defined normal
$5$-edge-coloring of $\tilde{G}.$ First, we establish compatibility within
conjoined paires, i.e. that the left superedge and the right superedge of such
pair are mutually compatible and also compatible with $\tilde{\sigma
}_{\mathrm{int}}.$

If $\mathcal{B}_{i-1}$ and $\mathcal{B}_{i}$ form a regular conjoined pair,
this implies that $d_{i-1}=d_{i}=1$ as illustrated by Figure \ref{Fig10}, and
it is clear from the figure that the left and the right superedge are
compatible for the chosen color combination, and that the same holds for all
color combinations inherited from $\sigma$ is established in Lemma
\ref{Lemma_parJedinica}. On the other hand, if $\mathcal{B}_{i-1}$ and
$\mathcal{B}_{i}$ form an odd conjoined pair, then $d_{i-1}=1$ and
$d_{i}\not =1$ as illustrated by Figure \ref{Fig_jednaJedinica}. For one
possible color combination inherited from $\sigma,$ this figure demonstrates
that the left and the right superedge of the pair is compatible mutually and
with $\tilde{\sigma}_{\mathrm{int}}.$ In Lemma \ref{Lemma_jednaJedinica} it is
established that the same holds for all possible color combinations.

Hence, we now have a sequence along the cycle $C$ of either singletons or
conjoined pairs which are all compatible within, and it remains to establish
their mutual compatibility. On each such "chunk" which consists of either only
one or of two superedges, the leftmost connector is left-side $\sigma
$-compatible and the rightmost connector is is right-side $\sigma
$-monochromatic (for singleton this is due to Observation \ref{Obs_leftR}, for
regular conjoined pair due to Lemma \ref{Lemma_parJedinica} and for odd
conjoined pair due to Lemma \ref{Lemma_jednaJedinica}). So, when the two such
"chunks" connect, the right side of the left "chunk" is right-side $\sigma
$-monochromatic and left side of the right "chunk" is left-side $\sigma
$-compatible, so the two chunks are compatible mutually and with
$\tilde{\sigma}_{\mathrm{int}}$ due to Lemma \ref{Lemma_borderCompatibility}.

It remains to prove that the constructed coloring $\tilde{\sigma}$ of
$\tilde{G}$ contains at least 18 poor edges. Notice that for the length $g$ of
the cycle $C$ it holds that $g\geq3.$ Recall that by Observation
\ref{Obs_rightPoor} each of the colorings $R$, $\bar{R},$ $R_{I}$ or $\bar
{R}_{I}$ contains $9$ poor edges. We shall call these four colorings
\emph{right} colorings, so it is sufficient to establish that at least two
superedges of $\tilde{G}$ are colored by a right coloring. If $\tilde{G}$ does
not contain an even-chain, then $\mathcal{B}_{i}$ is a singleton for every
$i\in\{0,\ldots,g-1\}$ and hence colored by a right coloring, the claim now
follows from $g\geq3.$ If $\tilde{G}$ contains precisely one conjoined pair,
then $g\geq3$ implies that $\tilde{G}$ contains at least one singleton. In
this case a singleton and the right superedge of a conjoined pair are colored
by a right coloring. Finally, if $\tilde{G}$ contains at least two conjoined
pairs, then the right superedge of each conjoined pair is colored by a right
coloring, and we are finished.
\end{proof}

Notice that Theorems \ref{tm_stari} and \ref{tm_main} provide a normal
$5$-edge-coloring of a superposition $\tilde{G}$ in all cases when a coloring
of the superposition can be constructed from a normal $5$-edge-coloring of an
underlaying snark $G$ by preserving colors of $\sigma$ outside the cycle $C.$
The only remaining case is the case of a cycle $C$ of odd length along which
superedges are connected to each other so that $p_{i}=(1,2,3)$ and $d_{i}=1$
for every $i.$ In this case an example is provided for which superposition
cannot be colored by preserving colors of $\sigma.$ Possibly, this might be
exploited to construct a counterexample to the Petersen Coloring Conjecture,
if it does not hold in general.

\section{Concluding remarks}

In this paper we consider superpositions of a snark $G$ along a cycle $C.$
Vertices of $C$ are superpositioned by one of two simple supervertices $A$ or
$A^{\prime}$ shown in Figure \ref{Fig_supervertex} and edges of $C$ are
superpositioned by a proper superedge $(P_{10})_{u,v}$ obtained from the
Petersen graph $P_{10}$ by removing two non-adjacent vertices $u,v$ from it
which is illustrated by Figure \ref{Fig_PetersenSuperedge}. A semiedge of the
left connector of $(P_{10})_{u,v}$ which connects to the semiedge of $A$
(resp. $A^{\prime}$) incident to $u$ is called the dock. In
\cite{SedSkrePaper1}, superpositions where $s_{1}^{l}$ is the dock for each
superedge superpositioned along $C$ are considered. Assuming that $G$ has a
normal $5$-edge-coloring $\sigma$, for an even cycle $C$ it is established
that $\sigma$ can be extended to a normal $5$-edge-coloring of a
superposition. In the case of an odd cycle $C$ such extension is not always
possible, i.e. the example of a superposition along an odd cycle is given such
that a normal $5$-edge-coloring which preserves coloring $\sigma$ outside $C$
does not exist.

In this paper we consider all remaining connections of superedges and
supervertices, i.e. superpositions where the dock is $s_{2}^{l}$ or $s_{3}%
^{l}$ for at least one superedge superpositioned along $C.$ For all such
superpositions we establish that a normal $5$-edge-coloring of $G$ can be
extended to the superposition while preserving coloring of $G$ outside $C.$
Moreover, our construction confirms the conjecture regarding the number of
poor edges proposed in \cite{SedSkrePaper1} for the considered class of
snarks. To be more precise, let $P_{10}^{\Delta}$ denote the graph obtained
from $P_{10}$ by truncating one vertex, the following conjecture is proposed
in \cite{SedSkrePaper1}.

\begin{conjecture}
Let $G$ be a bridgeless cubic graph. If $G\not =P_{10}$, then $G$ has a normal
$5$-edge-coloring with at least one poor edge. Moreover, if additionaly
$G\not =P_{10}^{\Delta}$, then $G$ has a normal $5$-edge-coloring with at
least $6$ poor edges.
\end{conjecture}

For all graphs in the class of superpositioned snarks considered in this
paper, we constructed a normal $5$-edge-coloring with at least $18$ poor edges
(see Theorem \ref{tm_main}), which supports the above conjecture.

\bigskip

\bigskip\noindent\textbf{Acknowledgments.}~~Both authors acknowledge partial
support of the Slovenian research agency ARRS program\ P1-0383 and ARRS
project J1-3002. The first author also the support of Project
KK.01.1.1.02.0027, a project co-financed by the Croatian Government and the
European Union through the European Regional Development Fund - the
Competitiveness and Cohesion Operational Programme.

\end{document}